\documentclass[amsart, 12pt]{amsart}
\usepackage{graphicx} 

\usepackage{amssymb,amsfonts,amsthm,amsmath,latexsym}
\usepackage{enumerate, caption, array, esint}

\usepackage{pgfplots}
\pgfplotsset{compat=1.15}
\usepackage{mathrsfs}
\usetikzlibrary{arrows}
\usepackage{subcaption}

\usepackage[margin=1.1in]{geometry}

\makeatletter
\newcommand*{\rom}[1]{\expandafter\@slowromancap\romannumeral #1@}
\makeatother

\theoremstyle{plain}
\newtheorem{thm}{Theorem}[section]
\newtheorem{cor}[thm]{Corollary}

\newtheorem{lem}[thm]{Lemma}
\newtheorem{prop}[thm]{Proposition}

\newtheorem{example}[thm]{Example}

\theoremstyle{definition}
\newtheorem{defn}[thm]{Definition}

\theoremstyle{remark}
\newtheorem{rem}[thm]{Remark}

\numberwithin{equation}{section}

\newcommand{\hide}[1]{{}}

\theoremstyle{plain}

\begin{document}

\title{The Hausdorff dimension of the set  where the Minkowski question mark function has infinite derivative}

\author{M. Pollicott}
%\footnote{Department of Mathematics, Warwick University, CV4 7AL-UK}}
%\affil[1]{Department, University, Address}

\address{Department of Mathematics, Warwick University, Coventy, CV4 7AL-UK}
\email{masdbl@warwick.ac.uk}
%\urladdr{http://homepage.com}

\date{\today}

\maketitle

\section{Introduction}

The famous Minkowski question mark function $Q: [0,1] \to [0,1]$ was defined by Minkowski in 1904
(at the third  ICM in Heidelberg) \cite{Mi}.  In particular,
given the continued fraction expansion of an irrational number
$$
x = \frac{1}{a_1 + \frac{1}{a_2 + \cdots + \frac{1}{a_{n-1} + \cdots }}}
\in (0,1)
$$
the image of $x$ under $Q$ is given by
$$
Q(x) = 2  \sum_{n=1}^\infty \frac{(-1)^{n+1}}{2^{a_1 + \cdots + a_n}}
$$
This definition is due to Denjoy \cite{De}. 
This is an example of a continuous monotone increasing function which has zero derivative almost everywhere. 
Furthermore, at every point the derivative cannot take any other finite value (i.e., every point
has derivative zero or  infinity, or the derivative will not exist).

\begin{figure}[h!]
\centerline{
\includegraphics[width=0.45\textwidth]{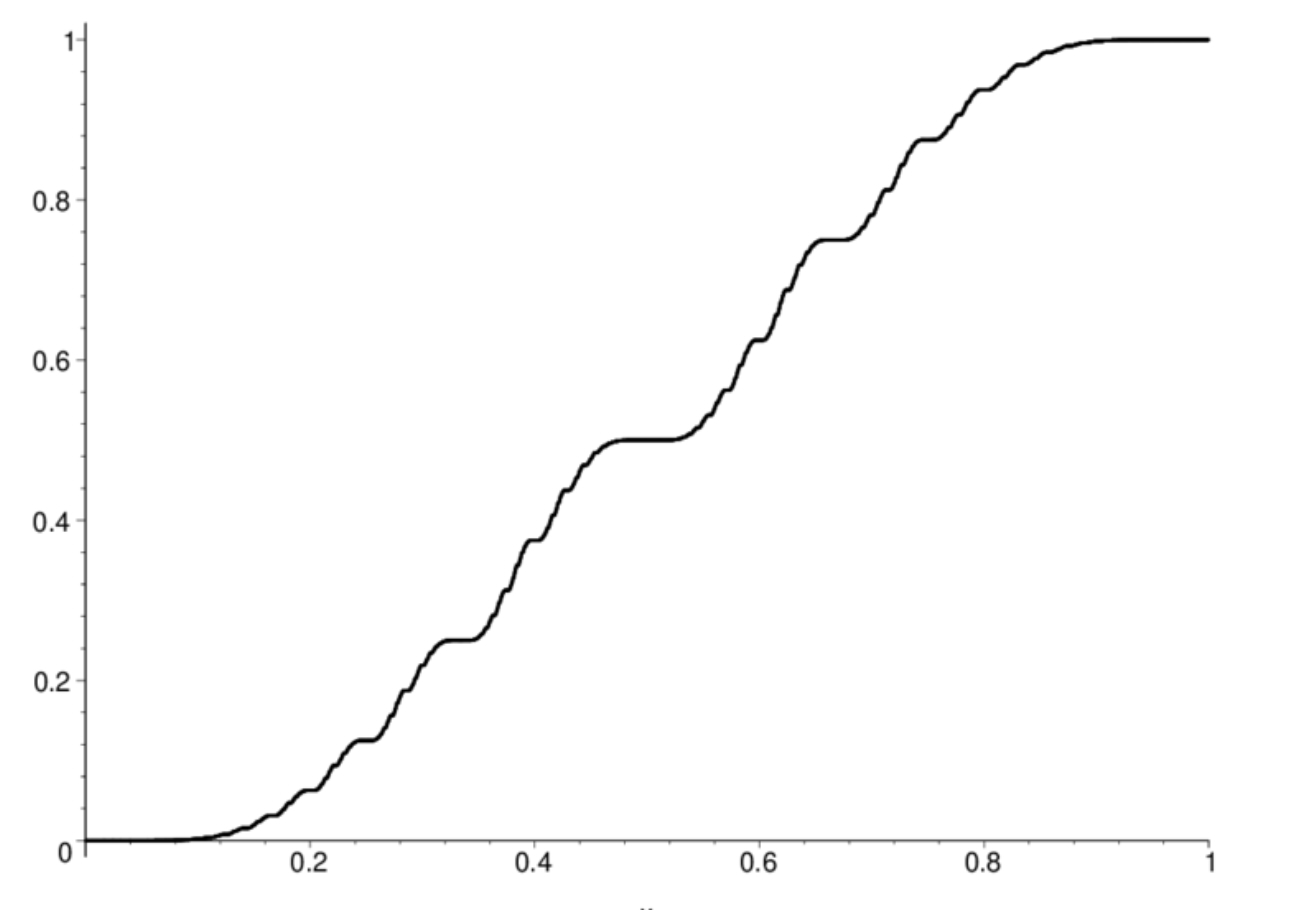} 
\includegraphics[width=0.45\textwidth]{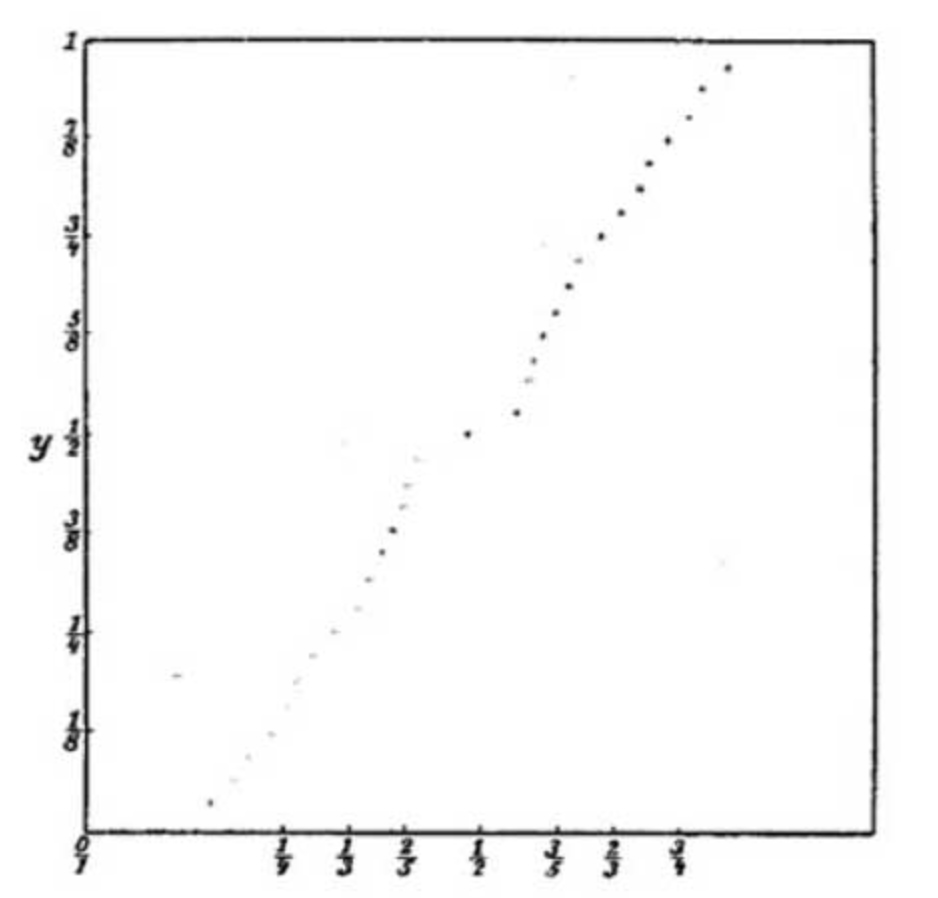} 
}
\caption{(i) The graph of the Minkowski function, (ii) The original plot from Minkowski's 1904 ICM talk}
\end{figure}

We will be interested in what happens on the complement  of this set. 

\begin{defn}
We let 
$\Lambda_\infty \subset [0,1]$ be the set if points for which the derivative exists and has  infinite  derivative 
\end{defn}

  In 
\cite{DKM} there is a characterization of points in $\Lambda_\infty$ in terms of continued fraction expansions.

By the above we know that $\Lambda_\infty$ has zero Lebesgue measure.
However, as was shown by 
Kinney \cite{Ki}
the set of points with infinite derivative has Hausdorff dimension 
$\dim_H(\Lambda_\infty) \in (0,1)$, i.e.,  strictly between between zero and unity. 
Our contribution is the following.

\begin{thm} We can estimate the numerical value of $\dim_H(\Lambda_\infty)$ by:
$$
\dim_H(\Lambda_\infty) = 
0.87471 63051 08211 14221 51529 04219 15975 77579 27289 75153 \ldots
%0.87471 63051 08211 14221 51529 04219 15975 77579 2728 \ldots
%0.874716305108211143577114190907246868457774 \ldots
$$
accurate to the $50$ decimal places given.
\end{thm}

\begin{rem}
Kesseboemer and Stratmann  \cite{KS} show the Hausdorff dimension of the set of points where the derivative does not exist (nor is infinity) has the same value for its dimension as $\Lambda_\infty$.
\end{rem}

In the literature there is a estimate accurate to  $13$ decimal places by Mantica \cite{Ma}
and an estimate accurate to  $35$ decimal places  by  Alkauskas \cite{Al}.  
Other early estimates include those by Lagarias (who placed the value in the range $[0.8746, 0.8749]$ \cite{La}) 
and Paradis, Viader and Bibiloni \cite{PVB} (who according to \cite{Al} revised their estimates 
to put  the value in the range $[0.874716, 0.874719]$).
Finally, there was an approximate estimate $\dim_H(\Lambda_\infty) \approx 0.875$ in \cite{TU}.

The main purpose of this note is to illustrate a simple alternative  method to get accurate and rigorous estimates. 

\section{A formula $\dim_H(\Lambda_\infty)$ }
In order to estimate $\dim_H(\Lambda_\infty)$  want  to express this quantity in terms of dynamically defined quantities.
We begin with the following classical map on the unit interval.

\begin{defn}
We define the \emph{Farey map} $F:[0,1] \to [0,1]$ by 
$$
F(x) = 
\begin{cases}
\frac{1}{1-x} &\hbox{ if } 0 \leq x \leq \frac{1}{2}\\
\frac{1-x}{x} &\hbox{ if } \frac{1}{2} \leq x \leq 1
\end{cases}
$$
\end{defn}

\begin{figure}[h!]
\centerline{
\includegraphics[width=0.5\textwidth]{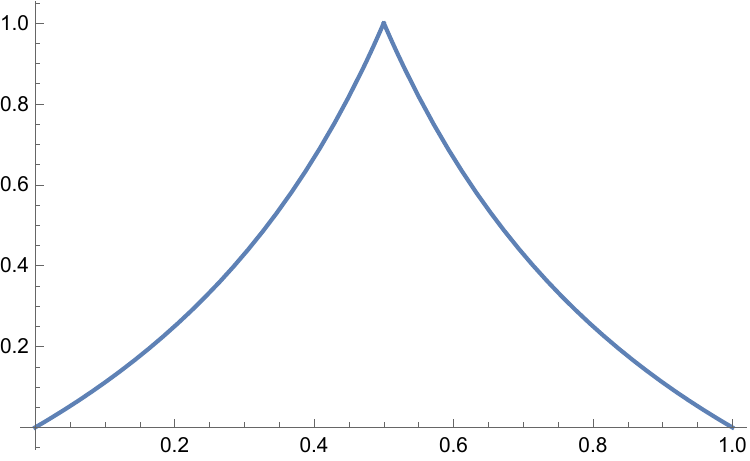} 
}
\caption{The graph of the Farey map}
\end{figure}

We let $\mu$ denote the measure of maximal entropy (or Parry measure) for $F$.  
%This corresponds to the $(\frac{1}{2}, \frac{1}{2})$-Bernoulli measure associated to the usual coding.  Alternatively, 
(Since Farey map is topologically conjugate to the usual tent map on the unit interval and then $\mu$ is the image of the Lebesgue under this conjugating map, which is the question mark function).
Associated to this measure is its entropy which is precisely $\log 2$.

  A second useful numerical quantity associated to this measure is the following:

\begin{defn}
We define the \emph{Lyapunov exponent} $\lambda(\mu)$ for the measure $\mu$ by
$$
\lambda(\mu) = \int \log |F'(x)| d\mu(x)
\left(
= 2\int_0^1 \log(1+x) d\mu(x)
\right)
$$
\end{defn}

Often in the dynamical approach to the Hausdorff dimension of appropriate sets  the value of the Hausdorff dimension 
turns out to be the ratio of the entropy to the Lyapunov exponent. 
Kinney \cite{Ki} and  Kesseboemer-Stratmann \cite{KS}
showed that this was indeed the case  for $\dim_{H}(\Lambda_\infty)$.

\begin{prop}[\cite{Ki},  \cite{KS}]
One can write 
$$\dim_{H}(\Lambda_\infty) = \frac{\log 2}{\lambda(\mu)}.$$
\end{prop}

In particular, this result translates the problem of estimating $\dim_{H}(\Lambda_\infty)$ into the problem of estimating 
$\lambda(\mu)$.

\section{Approach to estimating $\lambda(\mu)$}
We can now introduce a family of bounded   linear operators on the Banach space $C([0,1])$ of continuous functions on the unit interval with the usual supremum norm 
$$\|w\|_\infty = \sup_{0 \leq x \leq 1}|w(x)|.$$

\begin{defn}
Let $t \in \mathbb R$.
We define $\mathcal L_t: C([0,1]) \to C([0,1])$ by 
$$
\mathcal L_t w(x) = \frac{1}{2(1+x)^{2t}} \left( w \left( \frac{x}{1+x}\right) + w \left( \frac{1}{1+x}\right) \right) 
$$
where $w\in C([0,1])$.
\end{defn}
In particular, the  functions $T_1: x \mapsto \frac{x}{1+x}$ and $T_2: x \mapsto \frac{1}{1+x}$ arise as inverse branches to the usual Farey map.  
We can denote the \emph{spectral radius} of $\mathcal L_t$ by using the standard spectral radius formula:
\footnote{For simplicity, we formulate this for  $\mathcal L_t$ acting on continuous functions. However, in the proofs one uses the operator acting on $C^1$ functions and the fact that the spectral radii are the same for both Banach spaces.}
$$
\rho(t) := \lim_{n \to +\infty} \|\mathcal L_t^n\|_\infty^{1/n}.
$$
\begin{example}
For example,  when $t=0$ then $\mathcal L_0$ preserves the constant functions are eigenfunctions for the maximal eigenvalue  $1$, which is also the spectral radius, i.e.,  $\rho(0)=1$.
\end{example}

We can relate the spectral radius  $\rho(t)$ of $\mathcal L_t$  to the value $\lambda(\mu)$ by the following bounds.
\begin{prop}\label{distributionbd}
For $\epsilon > 0$ we  can bound 
$$
\frac{|\log \rho({\epsilon})|}{\epsilon}
\leq
\lambda(\mu)
\leq
 \frac{\log \rho({-\epsilon})}{\epsilon}
 \eqno(1)
$$
\end{prop}

In practice, the  values of $\rho(\pm \epsilon)$ maybe difficult to estimate directly from the definition.
However, we  can  implement more tractable bounds  using the following simple lemma.

\begin{prop}\label{functionbd}
For any $\epsilon > 0$ assume we have 
\begin{enumerate}
\item
a pair of 
real numbers $r_- > 1 > r_+$; and
\item a pair of 
 functions $w_-, w_+ \in C([0,1])$  with $w_-, w_+ > 0$  
 \end{enumerate}
satisfying 
$$
r_+ \leq 
\inf_{0 \leq x \leq 1} \frac{\mathcal L_{\epsilon}w_+(x)}{w_+(x)}
\hbox{ and }
\sup_{0 \leq x \leq 1} \frac{\mathcal L_{-\epsilon}w_-(x)}{w_-(x)}
\leq r_-
$$
then $r_+ \leq \rho(+\epsilon)
% \geq r_+
$ and 
$\rho(-\epsilon) \leq r_- $.
\end{prop}

Comparing   Proposition  \ref{distributionbd} 
and Proposition  \ref{functionbd} we can 
  deduce that
$$
\frac{|\log r_+|}{\epsilon}
%\leq 
%\frac{|\log \rho({\epsilon})|}{\epsilon}
\leq
\lambda(\mu)
%\leq
% \frac{\log \rho({-\epsilon})}{\epsilon}
 \leq \frac{|\log r_-|}{\epsilon}.\eqno(2)
$$

\section{Proof of Propositions  \ref{distributionbd} and  \ref{functionbd}}

We need to formulate some properties of the transfer operators $\mathcal L_{\pm \epsilon}$ used in the proof.  It is useful to  first consider it acting on the (smaller)  space $C^1([0,1]) \subset C([0,1])$ of $C^1$ functions 
which is a Banach space with the new norm
$$
\|w\| := \| w\|_\infty +  \| w'\|_\infty.
$$
We recall that the essential spectral radius $\rho_{ess}(t)$ is given by
$$
\rho_{ess}(t) = \lim_{n \to +\infty} \inf \left\{\|\mathcal L_t - K\| \hbox{ : } K \hbox{ is a compact operator} \right\}.
$$
The following lemma describes the spectrum of $\mathcal L_t$ on $C^1([0,1])$.

\begin{lem}[Spectrum of $\mathcal L_t$]\label{spectrum} 
Let $\epsilon < 1$.
%$4^{\epsilon + 1} = 2$.
%$2^{2\epsilon + 2} = 2$ and $$
For $-\epsilon \leq t \leq \epsilon$:
\begin{enumerate}
\item
The operator $\mathcal L_t: C^1([0,1[) \to C^1([0,1[)$   has an 
essential spectral radius $\rho_{ess}(t) < \rho(t)$;
\item 
The operator   $\mathcal L_t: C^1([0,1[) \to C^1([0,1[)$  has a simple maximal positive eigenvalue $\rho(t)$ with eigenfunction $h_t$ (i.e., $\mathcal L_th_t  = \rho(t) h_t$);
\item 
The dual operator   $\mathcal L_t^*: C^1([0,1[)^* \to C^1([0,1[)^*$  has a simple maximal positive eigenvalue $\rho(t)$ with eigenmeasure  $\mu_t$ (i.e., $\mathcal L_t^* \mu_t = \rho(t) \mu_t$); and 
\item The rest of the spectrum of $\mathcal L_t: C^1([0,1[) \to C^1([0,1[)$ lies inside a disk of radius  strictly smaller than 
$\rho(t)$.
\end{enumerate}
\end{lem}

This result is standard and well known in the case that $T_1$ and $T_2$ are strict contractions. 
(In this case, the value $\log \rho(t)$ is identified with the topological pressure of the function $-\log|T'|$.)
For the current case of $\mathcal L_t$ we outline the modified proof in the appendix.
\footnote{For the Farey map it is well known that when $t=1$ the  transfer operator does not have good spectral properties.  In particular, $\rho(1) = \rho_{ess}(1)$ and the function $h_1$ is not integrable, as is reflected in the absolutely continuous invariant probability being $\sigma$-finite, but not finite  finite.  However when, for example, 
 $t=0$ the situation is much better and the ambient measure $\mu_0$ is the measure of maximal entropy and $h_0=1$.  So for $t=0$, and nearby values, we can expect the operator to have good  spectral properties}

\begin{lem}[Properties of $\rho(t)$]\label{rho}
The spectral radius $\rho(t)$ has the following properties:
\begin{enumerate}
\item 
$\mathbb R \ni t \mapsto \log \rho(t)$ is  smooth and  monotone decreasing;
\item 
$\frac{d \log \rho(t)}{dt}|_{t=0} = -\lambda(\mu)$; and 
\item 
$\mathbb R \ni t \mapsto \log \rho(t)$ is   convex.
\end{enumerate}
\end{lem}

%Finally we come to the following proof.
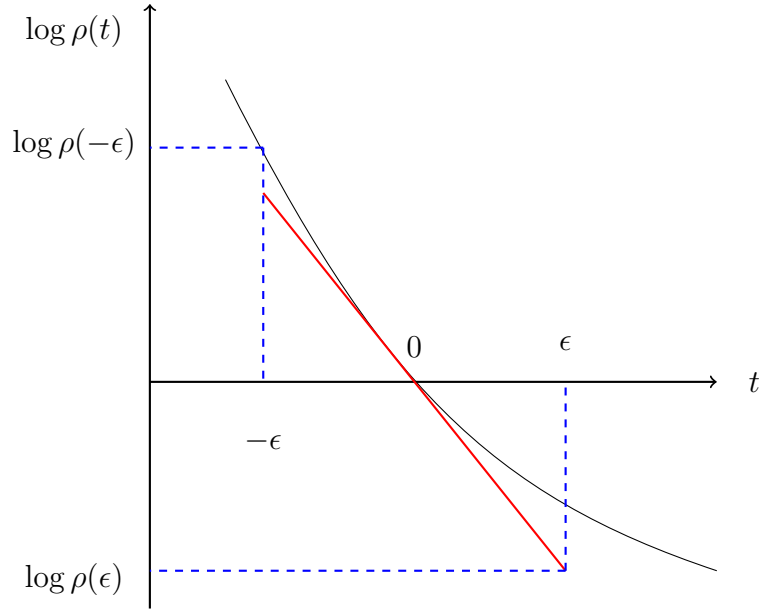
\begin{figure}[h!]
\centerline{
\begin{tikzpicture}[scale=0.5]
  \draw[thick,->] (0,0) -- (15,0);
    \draw[thick,->] (0,-6) -- (0,10);
\draw (2,8) .. controls (6,0) and (9,-3) .. (15,-5);
  \draw[thick,-, red] (3,5) -- (11,-5);
    \draw[thick,dashed, blue] (3,6.2) -- (3,0);
        \draw[thick,dashed, blue] (3,6.2) -- (0,6.2);
            \node at (-2,6.3) {$\log \rho(-\epsilon)$};
        \draw[thick,dashed, blue] (11,-5) -- (11,0);
                \draw[thick,dashed, blue] (11,-5) -- (0,-5);
        \node[below] at (3,-1) {$-\epsilon$};
               \node[below] at (7,1.5) {$0$};
                              \node at (16,0) {$t$};
                   \node[below] at (11,1.5) {$\epsilon$};
                \node[below] at (-2,10) {$\log \rho(t)$};
                  \node[below] at (-2,-4.5) {$\log \rho(\epsilon)$};
\end{tikzpicture}  
}
\caption{The proof of  Proposition  \ref{distributionbd}} 
\end{figure}

This result is again standard in the case that $T_1$ and $T_2$ are strict contractions. 
For the current case of $\mathcal L_t$ we outline the modified proof in the appendix.

\begin{proof}[Proof of Proposition  \ref{distributionbd}]
By the previous lemma we can write
$$\frac{d \log \rho(t)}{dt} \quad 
\begin{cases}
 \leq -\lambda(\mu) 
&\hbox{ if }   -\epsilon < t \leq 0\\
 \geq -\lambda(\mu) &\hbox{ if }   0 \leq t < \epsilon \\
 \end{cases}
 $$
 and integrating this derivative over the intervals $[-\epsilon, 0]$ and $[0, \epsilon]$, respectively, 
  gives the required bounds.
 This is illustrated in Figure 3.
\end{proof}

This brings us to the proof of the propositions.

\begin{proof}[Proof of Proposition  \ref{functionbd}]
This proof follows the lines argument \cite{PV}.
Using 
 the positivity of the operators  and iterating   the inequalities in the hypotheses gives that 
 $$r_-^n  w_{-}(x)  
\geq  r_-^{n-1}  (\mathcal L_{-\epsilon} w_-)(x) \geq
 r_- ^{n-2} (\mathcal L_{-\epsilon}^{2} w_-)(x) \geq
 \ldots
\geq (\mathcal L_{-\epsilon}^n w_-)(x)
\eqno(3a)
$$ 
and 
$$  r_+^n w_+(x) \leq
r_+^{n-1}  (\mathcal L_{\epsilon}  w_+)(x) \leq 
r_+^{n-2} (\mathcal L_{\epsilon}^2  w_+)(x) \let 
\cdots
\leq 
  (\mathcal L_{\epsilon}^n w_+)(x)
  \eqno(3b)
  $$
for all $0 \leq x \leq 1$ and all  $n \geq 1$.

Since by Lemma \ref{spectrum} the operators $\mathcal L_{\pm \epsilon}$ have  simple maximal eigenvalues $\rho({\pm \epsilon})$ (with eigenvectors $h_{\pm \epsilon}$ and eigendistributions $\mu_{\pm \epsilon}$ with $\mu_{\pm \epsilon}(h_{\pm \epsilon}) = 1$, without loss of generality) 
with no other eigenvalues on the unit circle we have that 
$$\lim_{n \to +\infty} \| \rho(\pm \epsilon)^{-n}\mathcal L_{\pm \epsilon}^nw_{\pm \epsilon} - 
\underbrace{\mu_{\pm \epsilon}(w_{\pm \epsilon})}_{=1}\|_\infty=0$$
and, in particular, 
$$
 \lim_{n \to +\infty} \sup_{0 \leq x \leq 1}\left| \mathcal L_\epsilon^n w_+(x)\right|^{1/n} 
 = \rho(\epsilon)
 \hbox{ and }
  \lim_{n \to +\infty} \inf_{0 \leq x \leq 1}\left| \mathcal L_{-\epsilon}^n w_-(x)\right|^{1/n} 
 = \rho(-\epsilon)
 \eqno(4)
$$
Moreover, we trivially have 
$$
 \lim_{n \to +\infty} \inf_{0 \leq x \leq 1}\left|  w_+(x)\right|^{1/n} 
 = 1 =   \lim_{n \to +\infty} \sup_{0 \leq x \leq 1}\left| w_-(x)\right|^{1/n} .
 \eqno(5)
$$
Thus, by (3a), (3b), (4) and (5) we have
 $$
 \begin{aligned}
 \frac{\rho(\epsilon)}{r_+}  
% \lim_{n \to +\infty} \sup_{0 \leq x \leq 1}\left| \mathcal L_\epsilon^n w_+(x)\right|^{1/n} 
\geq 1
\hbox{ and }
\frac{\rho(-\epsilon)}{r_-} &
% \lim_{n \to +\infty} \inf_{0 \leq x \leq 1}\left| \mathcal L_{-\epsilon}^n w_-(x)\right|^{1/n}
 \leq 1,\cr
%\rho(\epsilon):= & \lim_{n \to +\infty} \sup_{0 \leq x \leq 1}\left| \mathcal L_\epsilon^n w(x)\right|^{1/n} \geq r_+\cr
\end{aligned}
$$
as required.
%However, we know that 
%$\lim_{n \to +\infty } \|\mathcal L_{\pm \epsilon}^n w_{\pm \epsilon} \|^{1/n}= \rho(\mathcal L_{\pm \epsilon})$ 
%and trivially $\lim_{n \to +\infty } \|w_{\pm \epsilon} \|^{1/n} = 1$.
\end{proof}

\section{Implementing the estimates}
We can find candidate functions for $w_{-}$ and $w_{+}$ using standard ideas from numerical analysis.
More precisely, as in \cite{PV} let us fix $N \geq 2$ and choose
\begin{enumerate}
\item Chebychev nodes $\{x_n\}_{n=1}^N \subset [0,1]$ (where $x_n = \frac{1}{2} \left(1 +\cos\left(\frac{n - \frac{1}{2}}{N} \right) \right)$); and 
\item Lagrange polynomials $\{\ell\}_{n=1}^N$ (where $\ell_n(x) = \prod_{k=1}^N(x-x_n)/\prod_{k\neq n}(x_n - x_k)$).
\end{enumerate}

We can then associate the $N \times N$ matrices $M_{\pm \epsilon}$ with entries 
$$
M_{\pm \epsilon}(i,j) = (\mathcal L_{\pm \epsilon} \ell_i)(x_j) \hbox{ for } 1 \leq i,j \leq N
$$
For sufficiently large $N$ the matrix will have a maximal positive eigenvalue 
with eigenvector  $v_{\pm} = (v_{\pm, 1}, \cdots, v_{\pm, N})$.  Finally one can let 
$$
w_{\pm}(x) = \sum_{n=1}^N v_{\pm, n} \ell_n(x).
$$

We first give a simple implementation of these estimates to illustrate the method. 

\begin{example}[Simple bound]
Let  $\epsilon = \frac{1}{10^3}$ and  $N=6$.  
With this choice of $N$ the functions $w_-, w_+: [0,1] \to \mathbb R^+$ coming from the above constructions are polynomials of degree $5$ 

$$
\begin{aligned}
 w_+(x) &= 
(0.408614\ldots)
%183639952651027861117447371823140 - 
+
 (0.001189 \ldots)
 %04604564607073102706277937624081076 
 x + 
(0.001228 \ldots)
 %44534933113280004767284136776844701 
 x^2\cr
 &\quad - 
 (0.001309\ldots )
 %81504975842815285710651426539869424 
 x^3 
 + 
 (0.000869\ldots)
 %53869460608652879669267822865173571 
 x^4 - 
(0.000246\ldots)
 %028854217589677455225269915638087284 
 x^5
\end{aligned}
$$
and 
$$ 
\begin{aligned}
w_-(x) &= 
(0.407882 \ldots)
%797130703695893972564377825284734 
+ 
 (0.001184\ldots)
 %79744850006901356116756272819556904
  x - 
 (0.001218\ldots)
 %32227509081478574529755758178862648 
 x^2\cr  &\quad+ 
 (0.001296\ldots) 
 %85036969081371948041681316176889636 
 x^3 - 
 (0.000860 \ldots)
 %33009184670257536524874276003321854 
 x^4 + 
 (0.000243 \ldots)
 %340615667749756937426166371542101478 
 x^5
 \end{aligned}
$$
 These are illustrated in Figures 3 and Figure 4.  The plots of $w_+$ and $w_-$ are sufficiently close as to appear indistinguishable.
 
\bigskip
\noindent
One  can easily check that the associated values are
$r_- = 0.9992 0825 \ldots$ and $r_+= 1.00079246
%67
 \ldots $.

\begin{figure}[h!]
\centerline{
\includegraphics[width=0.4\textwidth]{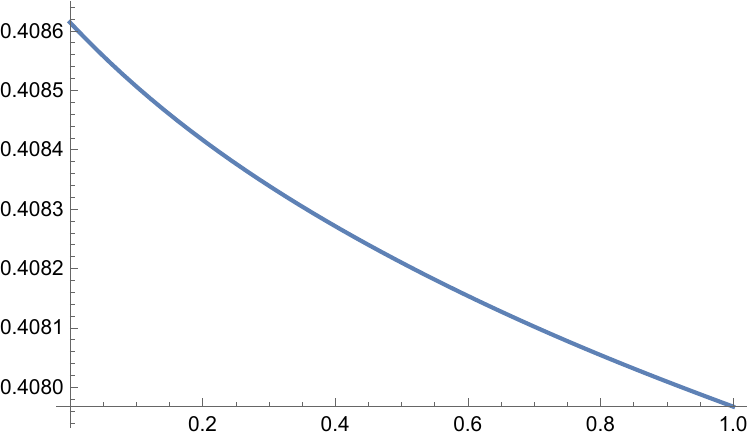} 
\hskip 0.75cm
\includegraphics[width=0.4\textwidth]{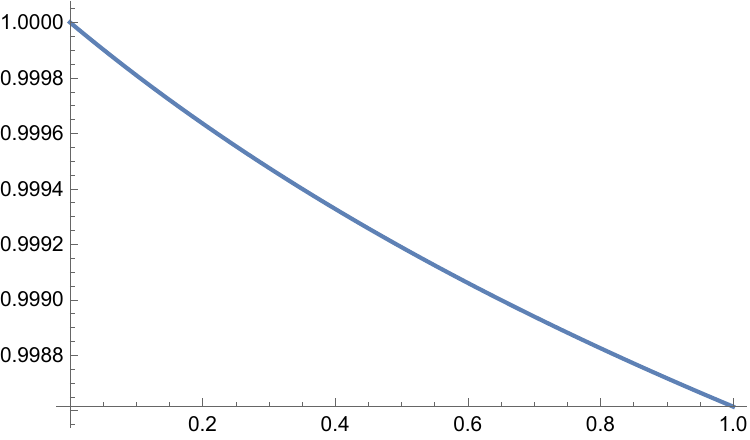} 
}
\caption{(i) The graph of the polynomial $w_-$ (ii) The graph of the ratio $(\mathcal L_{-\epsilon} w_-)/w_-$}
\end{figure}

\begin{figure}[h!]
\centerline{
\includegraphics[width=0.4\textwidth]{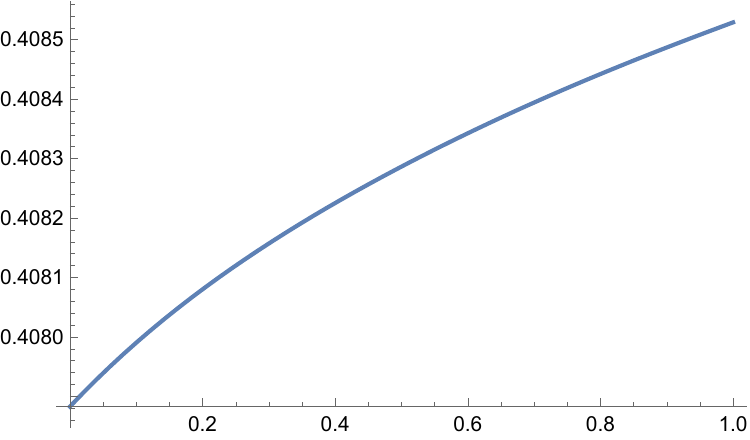} 
\hskip 0.75cm
\includegraphics[width=0.4\textwidth]{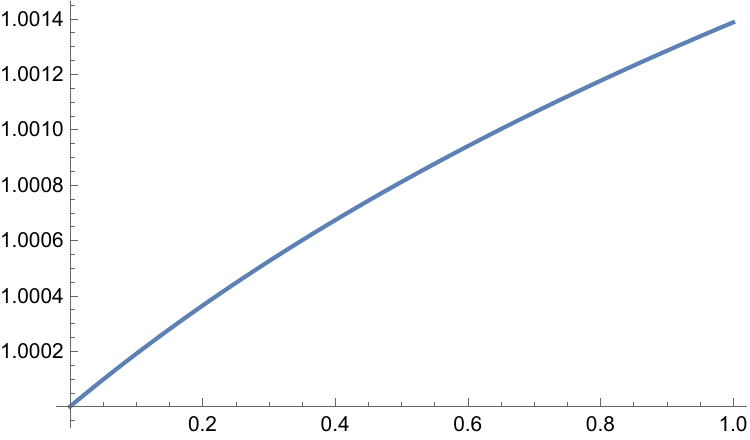} 
}
\caption{(i) The graph of the polynomial $w_+$ (ii) The graph of the ratio $(\mathcal L_{-\epsilon} w_+)/w_+$}
\end{figure}

\noindent
Working to $50$ decimal places one obtains from (2) an estimate for the Lyapunov exponent of 
$$
\begin{aligned}
{\bf 0.792}4 371 \ldots 
&\leq \lambda(\mu) 
\leq {\bf 0.792}5 251 \ldots
\end{aligned}
$$
This leads to basic bounds on the dimension of
$$
\begin{aligned}
{\bf 0.87}4169 \ldots
\leq \dim_H(\Lambda_\infty)
&\leq 
{\bf 0.87}5016 \ldots
\end{aligned}
$$
with the difference between the upper and lower bounds being $0.000847 \ldots$.
\end{example}

\begin{example}[Better bounds]
Let  $\epsilon = \frac{1}{120}$ and  $N=300$.  
Working to $200$ decimal places one obtains an estimate for the Lyapunov exponent of 
$$
\begin{aligned}
&
{\bf 0.792425128595489118191211515299891398889412782043800}62372638629709291\cr
&\leq \lambda(\mu) \leq  \cr
&
{\bf 0.792425128595489118191211515299891398889412782043800}93313729121852886\cr
\end{aligned}
$$
This leads to bounds on theHausdorff  dimension of $\Lambda_\infty$
$$
\begin{aligned}
&
{\bf 0.87471630510821114221515290421915975775792728975153} 223751895331660433\cr
&\leq \dim_H(\Lambda_\infty) \leq \cr
&
{\bf 0.87471630510821114221515290421915975775792728975153} 189597657442616397
\end{aligned}
$$
with the difference between the upper and lower bounds being $3.416  \times 10^{-52}$
\end{example}

\appendix 

\section{Proof of [Proofs of Lemmas \ref{spectrum} and \ref{rho}] }
We present the proofs of the two lemmas which were postponed from earlier.

\begin{proof}[Proof of Lemma \ref{spectrum}]
For part (1): Unfortunately, since the maps  $T_1$ and  $T_2$  are not strictly contracting we cannot directly apply the traditional results (e.g, from \cite{Ru}). It should be possible to apply some version of the  general results for systems that contract on average \cite{Pe}.
 However, we will sketch how  the proofs in the hyperbolic case  can be adapted to the present context.
 We can write
$$
\begin{aligned}
\mathcal L_t^2 w(x) &= |(T_1\circ T_1)'(x)|^t w(T_1 \circ T_2(x)) 
+  |(T_1\circ T_2)'(x)|^t w(T_1 \circ T_2(x)) \cr
&+  |(T_2\circ T_1)'(x)|^t w(T_1 \circ T_2(x)) 
+ |(T_2\circ T_2)'(x)|^t w(T_1 \circ T_2(x)) 
\end{aligned}
$$
where we recall that $T_1(x) = \frac{1}{1-x}$. and $T_1(x) = \frac{1-x}{x}$.
Observe that 
$$
\begin{aligned} 
 T_1\circ T_1(x) &= \frac{1+x}{2+x} \qquad T_1\circ T_2(x) = \frac{1+x}{1+2x} \cr
 T_2\circ T_2(x) &= \frac{x}{2+x} \qquad T_2\circ T_1(x) = \frac{1}{2+x} \cr
\end{aligned}
$$
and. therefore
$$
\begin{aligned}
\frac{1}{9} \leq |(T_1\circ T_1)'(x)| &= |(T_2\circ T_1)'(x)| = \frac{1}{(2+x)^2} \leq \frac{1}{4}\cr
\frac{1}{9} \leq |(T_2\circ T_2)'(x)| &= |(T_1\circ T_2)'(x)| = \frac{1}{(1+ 2x)^2}\leq 1. \cr
\end{aligned}
$$
Furthermore, 
$$
\begin{aligned}
|(T_1\circ T_1)''(x)| &= |(T_2\circ T_1)''(x)| = \frac{2}{(2+x)^3} \leq 1\cr
|(T_2\circ T_2)''(x)| &= |(T_1\circ T_2)''(x)| = \frac{4}{(1+ 2x)^2}\leq 4 \cr
\end{aligned}
$$

We can now bound 
$$
\left|\left(
\mathcal L_t^2 w\right)' (x)
\right|
 \leq  \frac{1}{2}\sum_{i,j=1}^2  \frac{|t|. |(T_i\circ T_j)''(x)|}{ |(T_i\circ T_j)'(x)|^{1-t}} |w(T_i \circ T_j(x))|
 +  \frac{1}{2}\sum_{i,j=1}^2 |(T_i\circ T_j)'(x)|^{t+1}|w'(T_i \circ T_j(x))|
$$
and deduce that
$$
\|(\mathcal L_t^2 w)'\|_\infty \leq
\underbrace{\left(\frac{5  |t|}{9^{1-t}}\right)}_{=C} \|w\|_\infty
+ \underbrace{\frac{1}{2}  \left( \frac{1}{4^{t+1}} + 1 \right)}_{=:\theta < 1}  \|w'\|_\infty.
$$
By iteration one can show that for $n \geq 2$,
$$
\|(\mathcal L_t^{2n} w)'\|_\infty \leq \frac{C}{1-\theta} \|w\|_\infty +  \theta^n\|w'\|_\infty
$$
(i.e., a Doeblin-Fortet-Ionescu-Tulcea-Lasota-Yorke inequality).  From here the proof of quasi-compactness follows by the standard approach involving compact operators coming from the anti-derivative (cf. \cite{HH}).

For parts (2) and (3): The existence of a simple maximal eigenvalue follows by first considering the 
map on the convex weak-star compact space $\mathcal M$ of probability measures defined by 
 $\mu \mapsto \mathcal L_t^*\mu/(\mathcal L_t^*\mu(1))$.
Then by the Schauder - Tychonoff. fixed point theorem there exists the fixed point $\mu_t \in \mathcal M$.
This can then be used to define the  cones 
$$
\mathcal C_\lambda  = \left\{
f \in C( [0,1],  \mathbb R^+)  \hbox{ : } f(x) \leq e^{\lambda |x-y|} f(y), \forall x,y \in [0,1] \hbox{ and } \mu_t(f) = 1
\right\}
$$
By Doeblin-Fortet-Ionescu-Tulcea-Lasota-Yorke inequality we see that $\mathcal L_t^n  \mathcal C_\lambda \subset \hbox{\rm int}(\mathcal C_{\lambda})$ for $n$ sufficiently large.
It then follows from the theory of Birkhoff metric on cones  and the contraction mapping principle that  there exists a Lipschitz function 
$h_\lambda \in C([0,1), \mathbb R)$ which is an eigenfunction of the maximal eigenvalue.  However, by quasi-compactness one can see that actually $h_\lambda\in C^1([0,1])$.
%and then using the fixed measure $\mu_t = \rho(t)^{-1}
%\mathcal L_t^* \mu_t$ to give a normalization for a cone of $C^1$ functions.

For part (4): We can assume for a contradiction that there exists an eigenfunction $k_t$ for an eigenvalue $\rho(t) r^{i\theta}$
with $0 < \theta < 2\pi$. However, then 
$$
\frac{1}{\rho(t)}|\mathcal L_t k_t(x)| \leq \frac{1}{2\rho(t)} \left| \sum_{i} |T_i'(x)|^t k_t \left( T_i x\right)\right| 
\leq  \frac{1}{2\rho(t)} \sum_{i} |T_i'(x)|^t \left| k_t \left( T_i x\right)\right| 
$$
and by comparing this  with $\frac{1}{\rho(t)}\mathcal L_t h_t(x) = h_t(x) > 0$ and using a simple convexity argument we can deduce that $h_t = k_t$ and therefore $e^{i\theta}=1$, giving a contradiction.
\end{proof}

\begin{proof}[Proof of Lemma \ref{rho}]
There is a routine thermodynamic proof for the usual transfer operators which we adapt to the present  setting \cite{Ru}, \cite{Ru+}.
Since $\rho(t)$ is an isolated eigenvalue it follows by analytic perturbation theory that $\rho(t)$ depends analytically on $t$ \cite{Ka}.

 To show (1) we  can follow (\cite{PP}, p.60) and differentiate (in $t)$ the eigenvalue equation $\mathcal L_t h_t = \rho(t)h_t$ to get
 $$
 \mathcal L_t \left(\frac{d h_t}{dt}\right) -  \mathcal L_t \left( \log |F'| h_t\right)  =     \frac{d \rho(t)}{dt} h_t  +   \rho(t) \frac{d h_t}{dt}
 $$
 (where we use that $|T_i'| = (1/|T'|)\circ T_i$ (for $i=1,2$) by the chain rule applies to $T\circ T_i(x) = x$).
We can then apply the dual eigenmeasure identity $\mathcal L_t^* \mu_t = \rho(t)\mu_t$ to write
$$
\underbrace{
\mu_t \left(\mathcal L_t \left(\frac{d h_t}{dt}\right) \right)
}_{=\rho(t)\mu_t\left(\frac{d h_t}{dt} \right)}
 -
 \underbrace{\mu_t\left(\mathcal L_t \left( \log |F'| h_t\right) \right)}_{= \rho(t)\mu_t(\log |F'| h_t)} =    
 \frac{d \lambda_t}{dt} \underbrace{\mu_t \left( h_t \right) }_{=1} +  \rho(t)  \mu_t \left( \frac{d h_t}{dt} \right).
$$
which after canceling (the first term on the left hand side with the last  term on the right hand side) gives  gives $\frac{d \rho(t)}{dt} = -\mu_t(\log |F'| h_t) < 0$ and thus 
$\frac{d (\log \rho)(t)}{dt} = \frac{d \rho(t)}{dt}  \frac{1}{\rho(t)} < 0$, i.e., the function $\log \rho(t)$ is monotone decreasing.
 In particular, when $t=0$ then 
$h_0=1$ and $\rho(0)=1$  and so this expression reduces to  $\frac{d \log \rho(t)}{dt}|_{t=0} = -\mu_0(\log |F'|) = -\lambda(\mu)$, as required.

 %follow \cite{Ru} and observe that since 
 %$$ 
 %\begin{aligned}
 %\rho(t) &= \lim_{n \to +\infty} \frac{1}{n} \log \mathcal L_t^n (1)(0) \cr
  %&= \lim_{n \to +\infty} \frac{1}{n} \log
  %\sum_{i_1, \cdots, i_n \in \{1,2\}} 
  %\left( \sum_{i_1, \cdots, i_n \in \{1,2\}} |(T_{i_1} \circ \cdots \circ T_{i_n}    )'(0)|^t \right) 
  %\end{aligned}
 %$$
% we can write 
% $$ 
 %\begin{aligned}
 %\rho'(t) &= \lim_{n \to +\infty} \frac{1}{n} \frac{
 % \left( \sum_{i_1, \cdots, i_n \in \{1,2\}}
  %\log( |T_{i_1} \circ \cdots \circ T_{i_n}    )'(0)|).
  % |(T_{i_1} \circ \cdots \circ T_{i_n}    )'(0)|^t \right)}{ \left( \sum_{i_1, \cdots, i_n \in \{1,2\}} |(T_{i_1} \circ \cdots \circ T_{i_n}    )'(0)|^t \right)} < 0
 % \end{aligned}
 %$$
 The estimate on the second derivative follows the argument in (\cite{PP}, p.60-61) and gives
 $$
\frac{d^2 \log \rho(t)}{dt^2}  = \lim_{n \to +\infty}  \frac{1}{n} \int \left(
-\log |(T^n)'(x)| + \int \log |(T^n)'|d\mu_t
\right)^2 
d\mu_t \geq 0
$$
 This shows the function is convex.
 \end{proof}

\end{document}